\documentclass[12pt]{amsart}
\usepackage[all]{xy}
\usepackage{amsfonts, amsthm, amssymb, amscd, amstext, amsmath}
\usepackage{url, xr}
\usepackage{bbm}
\usepackage{stmaryrd}
\usepackage[mathscr]{euscript}
\usepackage{mathrsfs}
\usepackage{dsfont}
\usepackage{textcomp}
\usepackage{enumitem}

\usepackage[all]{xy}
\newtheorem{defn}{Definition}[section]

\newtheorem{prop}[defn]{Proposition}
\newtheorem{thm}[defn]{Theorem}
\newtheorem{cor}[defn]{Corollary}
\newtheorem{rmk}[defn]{Remark}

\def\C{{\mathbb C}} 
\def\P{{\mathbb P}} 
\def\Q{{\mathbb Q}}

\def\CC{{\mathcal C}}

\def\FF{{\mathcal F}} 
 
\def\HH{{\mathcal H}} 

\def\OO{{\mathcal O}}

\def\XX{{\mathcal X}} 
 
\def\ZZ{{\mathcal Z}} 

\def\mapright#1{\mathop{\vbox{\ialign{
                ##\crcr
    ${\scriptstyle\hfil\;\;#1\;\;\hfil}$\crcr
 \noalign{\kern-1pt\nointerlineskip}
    \rightarrowfill\crcr}}\;}}

\def\mapleft#1{\mathop{\vbox{\ialign{
                ##\crcr
    ${\scriptstyle\hfil\;\;#1\;\;\hfil}$\crcr
    \noalign{\kern-1pt\nointerlineskip}
    \leftarrowfill\crcr}}\;}}

\def\mapdown#1{\Big\downarrow
         \rlap{$\vcenter{\hbox{$\scriptstyle#1$}}$}}

\def\codim{\mathop{\rm codim}\nolimits}
\def\dim{\mathop{\rm dim}\nolimits}
\def\coker{\mathop{\rm coker}\nolimits}
\def\Gr{\mathop{\rm Gr}\nolimits}
\def\Hom{\mathop{\rm Hom}\nolimits}
\def\im{\mathop{\rm im}\nolimits}
\def\max{\mathop{\rm max}\nolimits}
\def\IC{\mathop{\rm IC}\nolimits}
\def\Perv{\mathop{\rm Perv}\nolimits}
\def\Sing{\mathop{\rm Sing}\nolimits}
\def\Supp{\mathop{\rm Supp}\nolimits}

\begin{document}
\title{Cohomology of complete intersections of quadrics}
\author{J. Nagel}	
\address{Institut de Math\'ematiques de Bourgogne, UMR 5584, 9 Avenue Alain Savary, BP 47870, 21078 Dijon, France}
\begin{abstract}
We show that the variable cohomology of a regular quadric bundle arising from a linear system of quadrics can be identified with the intersection cohomology of a double covering. As a consequence we show that the middle cohomology of a general smooth complete intersection of four quadrics in an odd-dimensional projective space is isomorphic to the cohomology $H^3(\tilde Z_0)$ of a resolution of singularities of a double solid. 	
\end{abstract}	

\dedicatory{To the memory of Alberto Collino}

\maketitle

\section{Introduction}

The geometry of complete intersections of quadrics in complex projective space have been studied extensively. Deligne's formula for the Hodge numbers of complete intersections in projective space implies that if $$X=V(Q_0,\ldots,Q_r)\subset\P^{n+1}$$ is a smooth complete intersection of quadrics, the level
$$
\ell = \{\max|p-q|, h^{p,q}(X)=0\}
$$
of the Hodge structure on the middle cohomology $H^{n-r}(X,\Q)$ equals $r$ if $n$ is even, and $r-1$ if $n$ is odd. Hence the generalised Hodge conjecture predicts that there should exist a smooth projective variety $Y$ of dimension $r$ or $r-1$ (depending on the parity of $n$)  and a correspondence $\Gamma$ from $Y$ to $X$ such that 
$$
\Gamma_*:H^{r-i}(Y,\Q)\to H^{n-r}(X,\Q)
$$
is surjective if $n\equiv i (2)$, $i\in\{0,1\}$. It suffices to prove the surjectivity of the induced map to the {\em variable cohomology} $H^*_{\rm var}(X)$, the cokernel of the restriction map from $H^*(\P^{n+1})$ to $H^*(X)$. 

In the case of quadrics there is a natural correspondence available. Let us describe it first in case $n=2m$ is even. A smooth quadric $Q\subset\P^{2m+1}$ contains two families of $m$-planes. If we choose a smooth quadric $Q_0$ containing $X$ and intersect the difference $\Lambda_t'-\Lambda_t''$ of two $m$-planes from the different families with $X$, we obtain a family of algebraic cycles $\{\ZZ_t\}$ of dimension $m-r$ on $X$. The base of this family is a double covering $Z_0$ of $\P^r$ that is ramified over the subset $\Delta\subset\P^r$ that parametrises singular quadrics in the linear system. We obtain a map (the cylinder homomorphism) from $H_r(Z_0)$ to $H_m(X)$ by mapping the class $[\gamma]$ of a topological $r$-cycle on $Z_0$ to the class of the topological $m$-cycle on $X$ swept out by the cycles $\{\ZZ_t\}_{t\in\gamma}$. If $X$ is smooth, we obtain a map $\Gamma_{0,*}:H^r_{\rm var}(Z_0)\to H^{2m-r}_{\rm var}(X)$ by Poincar\'e duality and passage to variable cohomology. Here $H^*_{\rm var}(Z_0)$ is the cokernel of the pullback map $H^*(\P^r)\to H^*(Z_0)$ coming from the double covering.

In the case $n=2m-1$ there is a similar construction, but in this case we have to start with a nodal quadric containing $X$ to obtain the two families of $m$-planes, and we obtain a family of algebraic cycles parametrised by a variety $Z_1$ of dimension $r-1$, a double covering of the discriminant $\Delta$, and a correspondence
$$
\Gamma_{1,*}:H^{r-1}_{\rm var}(Z_1)\to H^{2m-r-1}_{\rm var}(X).
$$

The expected codimension of the singular locus of the discriminant $\Delta$ is three. Hence the double coverings $Z_0$ and $Z_1$ are smooth if we consider general linear sytems of quadrics of dimension $r\le 2$ (pencils and nets). In these cases, the correspondences $\Gamma_{i,*}$ have been studied in detail \cite{Beauville,Mukai,O'Grady,Reid,Tyurin}. In all of these cases $\Gamma_{i,*}$ is an isomorphism (with $\Q$-coefficients).\footnote{Beauville's result even holds with integer coefficients.}

For linear systems of quadrics of dimension $r\ge 3$ the varieties $Z_i$ become singular. In this case the cohomology of $Z_i$ carries a mixed Hodge structure with an increasing weight filtration $W_{\bullet}$ and the subspace $W_{r-i-1}H^{r-i}(Z_i,\Q)$ belongs to the kernel of $\Gamma_{i,*}$, so we cannot expect injectivity of the correspondence. Terasoma \cite{Terasoma} proved generic surjectivity of the correspondences $\Gamma_{i,*}$ using a degeneration argument. In my unpublished habilitation thesis \cite{HDR} I simplified this method and gave precise conditions on surjectivity and injectivity based on the behaviour of the Leray spectral sequence for the associated quadric bundle $f:\XX\to\P^r$. 
	
The goal of this paper is to undertake a systematic study of the behaviour of the correspondences $\Gamma_{i,*}$ using perverse sheaves and intersection cohomology. We define  correspondences 
$$
\Gamma_{i,\#}:IH^{r-i}(Z_i)\to IH^{2m-r-i}(X) = H^{2m-r-i}(X)
$$
on the intersection cohomology groups and show that 
\vfill\eject

{\bf Theorem.}
\begin{enumerate}
\item[a)] There is an induced map 
$$
\Gamma_{i,\#}:IH^{r-i}_{\rm var}(Z_i)\to H^{2m-r-i}_{\rm var}(X).
$$
This map is an isomorphism for $n\equiv i(2)$.
\item[b)] Let $\sigma:\tilde Z_i\to Z_i$ be a desingularisation of $Z_i$. For all $k$ and $r$  there exist a correspondence $\tilde\Gamma_{i,*}:H^k(\tilde Z_i)\to H^{k-2m-2r}(X)$ and a commutative diagram
$$
\xymatrix{H^k(Z_i) \ar[r]^{\omega} \ar[dr]^{\Gamma_{i,*}} & IH^k(Z_i) \ar[r]^{\sigma^{\#}} \ar[d]^{\Gamma_{i,\#}} & H^k(\tilde Z_i) \ar[dl]^{\tilde\Gamma_{i,*}} \\
& H^{k+2m-2r}(X). &
}
$$
\end{enumerate}

The two main ingredients of the proof are the {\em Radon transformation}, which gives the precise form of the decomposition theorem for the quadric bundle $f:\XX\to\P^r$ associated to $X$ and a weak functoriality result for intersection cohomology \cite{BBFGK} that enables us to construct the correspondences $\Gamma_{i,\#}$ and to show the commutativity of the diagram.

The first interesting case where our methods apply is $r=3$ (webs of quadrics). We study the case of complete intersections of four quadrics in odd-dimensional projective space in detail. 
In this case $Z_0$ is a double solid (a double covering of $\P^3$ ramified along a nodal surface). There are three ways to replace the mixed Hodge structure on $H^3(Z_0)$ by a pure Hodge structure of weight three: we can either take the weight graded part $\Gr^3_W H^3(Z_0)$, the intersection cohomology group $IH^3(Z_0)$ or the cohomology of the desingularisation $\tilde Z_0\to Z_0$ obtained by blowing up the nodes. We show that all three are isomorphic; in particular, we obtain an isomorphism 
$$
\tilde\Gamma_{0,*}:H^3(\tilde Z_0)\to H^{2m-3}(X)
$$
that identifies the middle cohomology of $X$ with the cohomology of a smooth projective threefold. this result can be seen as an analogue of the classical results for $r=1$ and $r=2$ mentioned before.\footnote{It follows from the results of Addington \cite{Addington} that $H^{2m-3}(X)$ is isomorphic to the cohomology $H^3(\hat Z_0)$ of a small resolution of $Z_0$. But $\hat Z_0$ is not a K\"ahler manifold.}   

\medskip	
{\bf Notation and conventions.} Homology and cohomology are taken with $\Q$-coefficients (unless stated otherwise). If $X$ is a complex algebraic variety we write $d_X = \dim_{\C}X$.	
	
\medskip{\bf Acknowledgments.} It is a pleasure to thank Nick Addington, Damien M\'egy and Chris Peters for helpful conversations. This work was supported by ANR HQDIAG (contract ANR-21-CE40-0015). The IMB receives support from the EIPHI Graduate School (contract ANR-17-EURE-0002).

\section{The correspondence on cohomology}

Let $X=V(Q_0,\ldots,Q_r)\subset\P^{n+1}$ be a smooth complete intersection of quadrics. We denote by 
$$
\XX = \{(x,\lambda)\in\P^{n+1}\times\P^r, \sum_i \lambda_i Q_i(x)=0 \}  \mapright{f} \P^r
$$
the associated quadric bundle over $\P^r$. Note that the {\em variable cohomology groups}
\begin{eqnarray*}
	H^{n+r}_{\rm var}(\XX,\Q) & = & \coker(H^{n+r}(\P^{n+1}\times\P^r,\Q)\to H^{n+r}(\XX,\Q)) \\
	H^{n-r}_{\rm var}(X,\Q)) & = & \coker(H^{n-r}(\P^{n+1},\Q)\to H^{n-r}(X,\Q)).
\end{eqnarray*}
are related by an isomorphism of Hodge structures
$H^{n+r}_{\rm var}(\XX,\Q)\cong H^{n-r}_{\rm var}(X,\Q)(-r)$; cf. \cite{Terasoma}.

The quadric bundle $\XX\to\P^r$ is obtained by pulling back the universal family of quadrics $\XX_{\rm univ}\to\P^N = \P H^0(\P^N,\OO_{\P}(2))$. 
Recall that the corank of a quadric $s$ is the corank $n+2-{\rm rank}(A_s)$ of the associated symmetric matrix. Write
$
\Delta_i^{\rm univ} = \{[s]\in\P^N| {\rm corank}(A_s)\ge i\}.
$
Then
\begin{itemize}
	\item $\Sing(\Delta_i^{\rm univ}) = \Delta_{i+1}^{\rm univ}$
	\item $\codim(\Delta_i^{\rm univ}) = {{i+1}\choose 2}$
\end{itemize}
so we get a natural stratification of the base of the universal family by locally closed subsets $U_i^{\rm univ} = \Delta_i^{\rm univ}\setminus\Delta_{i+1}^{\rm univ}$.

\begin{defn}
\label{def:regular}	
Let $S\subset\P^N$ be a smooth quasi--projective subvariety.	
We say that a quadric bundle $f:\XX\to S$ obtained by pullback from the universal family to $S$ is {\em regular} if the total space $\XX$ is smooth and if $S$ is transversal to the stratification by corank, i.e., if
\begin{itemize}
	\item $U_i^{\rm univ}\cap S$ is smooth for all $i$
	\item $\codim_S(\Delta_i^{\rm univ}\cap S) = {{i+1}\choose 2}$ for all $i$.
\end{itemize}
 \end{defn}


\subsection{Even relative dimension}

Let $f:\XX\to S$ be a regular quadric bundle of even relative dimension $2m$. Let $h:F_m(\XX/S)\to S$ be the relative Fano scheme of $m$--planes in the fibers of $\XX\to S$, and let 
$$
\xymatrix{
F_m(\XX_S/S) \ar[r]^{h} \ar[d]^{h'} &  S \\
 Z_0 \ar[ru]^{\pi} &
}
$$
be the Stein factorisation.	For technical reasons it is convenient to choose a subscheme $F\subset F_m(\XX_S/S)$ such that $h'|_F:F\to Z_0$ is a finite morphism. We then consider the pullback
$\Gamma_0 = \FF\times_{F_m(\XX_S/S)}F$ 
of the the universal family $\FF$ of $m$--planes over $F_m(\XX_S/S)$ to $F$. In this way we obtain a commutative diagram
$$
\xymatrix{
\Gamma_0 \ar[r] \ar[d] & \FF \ar[r] \ar[d] &  \XX \ar[d] \\
F \ar[r] \ar[dr] & F_m(\XX/S) \ar[d] \ar[r] & S \\
& Z_0. \ar[ru] &
}
$$
We denote the composed maps from $\Gamma_0$ to $Z_0$ and $\XX_S$ by $p$ and $q$. With these notations, we have a commutative square
$$
\begin{array}{ccc}
\Gamma_0 & \mapright{q} & \XX \\
\mapdown{p} & & \mapdown{f} \\
Z_0 & \mapright{\pi} & S.
\end{array}
$$
The correspondence $\Gamma_0$ induces a map 
$$
\Gamma_{0,*}:H^k(Z_0,\Q)\to H^{k+2r}(\XX,\Q)
$$
obtained by taking the composition of the maps
$$
H^k(Z) \mapright{p^*} H^k(\Gamma_0) \mapright{\cap[\Gamma_0]} H_{2d_{\Gamma_0}-k}(\Gamma) \mapright{q_*} H_{2d_{\Gamma_0}-k}(\XX) = H^{k+2m}(\XX)
$$
where we used that $\Gamma_0$ and $\XX$ are compact, $\XX$ is smooth and $d_{\XX}-d_{\Gamma}=m$.

The map $\Gamma_{0,*}$ is induced by a morphism in the derived category of bounded complexes of sheaves of $\Q$--vector spaces. Recall that if $f:X\to Y$ is a proper morphism of algebraic varieties and $D_X$, $D_Y$ denote the dualising sheaves on $X$ and $Y$
there exist maps 
\begin{eqnarray*}	
\alpha_f:\Q_Y\to Rf_*\Q_X \\
\beta_f:Rf_*D_X\to D_Y
\end{eqnarray*}
such that $\alpha_f$ induces the pullback map on cohomology
$$
f^*:H^k(Y,\Q)\to H^k(X,\Q) 
$$
and $\beta_f$ induces the pushforward map on Borel--Moore homology
$$
f_*:H_k^{\rm BM}(X,\Q) = H^{-k}(X,D_X)\to H_k^{\rm BM}(Y,\Q). 	
$$
Applying this to our situation, we obtain a map 
$\alpha_p:\Q_Z\to Rp_*\Q_{\Gamma_0}$
that induces
$$
a:R\pi_*\Q_Z\to R\pi_*Rp_*\Q_{\Gamma_0} = Rf_*Rq_*\Q_{\Gamma_0}
$$
and a map
$\beta_q:Rq_*D_{\Gamma}\to D_{\XX}$.
If we compose this with the map
$\Q_{\Gamma}\to D_{\Gamma}[-2d_{\Gamma}]$
that corresponds to the fundamental class $[\Gamma]\in H_{2d_{\Gamma}}^{\rm BM}(\Gamma,\Q) = \Hom(\Q_{\Gamma},D_{\Gamma}[-2d_{\Gamma}])$
we obtain
$$
b:Rq_*\Q_{\Gamma}\to D_{\XX}[-2d_{\Gamma}].
$$
Hence we obtain a morphism 
$$
\Gamma_{0,*}:R\pi_*\Q_{\Gamma}\mapright{a}Rf_*q_*\Q_{\Gamma}\mapright{Rf_*(b)}Rf_*D_{\XX}[-2d_{\Gamma}]=Rf_*\Q_{\XX}[2m].
$$ 
in the derived category that induces the action of the correspondence $\Gamma_0$ on cohomology.

We also obtain a homomorphism of sheaves
$$
\Gamma_{0,*}:\pi_*\Q\to R^{2m}f_*\Q.
$$
We can extract two rank one local systems from this situation.
The map $\pi:Z_0\to\P^r$ is a double covering that is ramified along the discriminant locus $\Delta\subset S$ of $f$. Write $U=S\setminus\Delta$. The unramified double covering $\pi_U:\pi^{-1}(U)\to U$ induces a map 
$\Q_U\to\pi_{U,*}\Q$
whose cokernel is the rank one local system $L_0\cong(\pi_{U,*}\Q)^{-}$ (the anti-invariant part of the action of the involution of the double covering). The commutative diagram
$$
\xymatrix{
\XX \ar[r] \ar[d]^{f} & \P^{n+1}\times S \ar[dl]^{\varphi} \\
S & 
}
$$
gives a restriction map $R^{2m}\varphi_*\Q|_U\to R^{2m}f_*\Q|_U$
whose cokernel is the rank one local system $M_0 = (R^{2m}f_*\Q|_U)_{\rm var}$.

\begin{prop}
\label{prop:iso loc even}	
There exists a Zariski open subset $V\subset U$ such that the restriction of $\Gamma_0$ to $V$ induces an isomorphism of rank one local systems
$$
\Gamma_{0,*}|_V:(L_0)|_V\mapright{\sim} M_0|_V
$$
\end{prop} 

\begin{proof}
Recall that in order to define $\Gamma$ we chose $F\subset F_m(\XX/\P^r)$ such that $h'|_F:F\to Z$ is a finite morphism, say of degree $d$. 	
There exists a Zariski open subset $V\subset U$ such that $h'$ is unramified of degree $d$ over $\pi^{-1}(V)$. Choose $s\in V$ and write $\pi^{-1}(s) = \{s',s''\}$. The fiber of $p\circ h'$ over $s'$ consists of $m$--planes $\{\Lambda'_1,\ldots,\Lambda'_d\}$ that all have the same cohomology class $\lambda'$, and over $s''$ we get $m$--planes belonging to the second family that all have the same cohomology class $\lambda''$.
The induced map on the stalks
$$
(\pi_*\Q)_s = \Q[s']\oplus\Q[s'']\to H^{2m}(\XX_s,\Q)=\Q[\lambda']\oplus\Q[\lambda''] 
$$
maps the generator $[s']-[s'']$ of $L_s$ to $d$ times the generator $\lambda'-\lambda''$ of $M_s$. 
\end{proof}
 
\subsection{Odd relative dimension}

If $n=2m-1$ is odd, we have a similar situation, but this time we have to consider quadrics that have one ordinary double point. A singular quadric $Q$ with one node $p$ is a cone with vertex $p$ over a smooth quadric  $Q'\subset\P^{2m-1}$, hence it contains two families of $m$--planes (obtained by taking the span of the ($m-1$)--planes on $Q'$ with $p$).  

Write $\XX_{\Delta} = f^{-1}(\Delta)\subset\XX$. The map from the relative Fano scheme $F_m(\XX_{\Delta}/\Delta)\to\Delta$ factors through a double covering $\pi:Z\to\Delta$ that is ramified over the locus of quadrics of corank 2. As before we obtain a diagram
$$
\begin{array}{ccccccc}
\Gamma & \mapright{} & F_m(\XX_{\Delta}/\Delta) & \mapright{} & \XX_{\Delta} & \mapright{} & \XX \\
\mapdown{} & & \mapdown{} & & \mapdown{} & & \mapdown{f}	\\
F & \mapright{} & Z & \mapright{\pi} & \Delta & \mapright{i} & \P^r. 
\end{array} 	
$$
With these notations, we obtain a commutative diagram
$$
\begin{array}{ccccc}
\Gamma_1 & \mapright{q_{\Delta}} & \XX_{\Delta} & \mapright{\iota} & \XX \\
\mapdown{p} & & \mapdown{f_{\Delta}} & & \\
Z_1 & \mapright{\pi} & \Delta. & &
\end{array}
$$
Write $q = \iota\circ q_{\Delta}:\Gamma_1\to\XX$. Then $\Gamma_1$ is a degree $m$ correspondence from $\Delta$ to $\XX$, and by similar arguments one obtains the analogue of Proposition \ref{prop:iso loc even} in the case of odd relative dimension.
\begin{prop}
\label{prop: iso loc odd}	
There exists a Zariski open subset $V\subset\Delta$ such that $\Gamma_{1,*}|_V$ induces an isomorphism between the rank one local systems $L_1 = (\pi_*\Q|_V)^-$ and $M_1 = (R^{2m}f_*\Q|_V)_{\rm var}$. 	
\end{prop}	

\section{The correspondence on intersection cohomology}

\subsection{Perverse sheaves and intersection cohomology}

We briefly recall some basic facts on perverse sheaves and intersection cohomology, mainly in order to fix the notation (since various ways of shifting intersection complexes exist in the literature). Good references are \cite{BBD} and \cite{KiehlWeissauer}.

Let $S$ be a complex algebraic variety. We denote the bounded derived category of constructible sheaves of $\Q$--vector spaces on $S$ by $D^b_c(\Q_S)$. We say that $P\in D^b_c(\Q_S)$ satisfies the {\em support condition} if $\dim\Supp(\HH^{-i}(P))\le i$ for all $i$. The full subcategory of complexes that satisfy the support condition is denoted by  ${^p}D^{\le 0}\subset D^b_c(\Q_S)$. We say that $P$ satisfies the {\em cosupport condition} if its Verdier dual $D_S(P)$ satisfies the support condition. The corresponding full subcategory is denoted by ${^p}D^{\ge 0}$. These subcategories define a {\em perverse t-structure} $({^p}D^{\ge 0},{^p}D^{\le 0})$ on $D^b_c(\Q_S)$ (we use the middle perversity). The heart of this $t$-structure is the abelian category of perverse sheaves 
$\Perv(S) = {^p}D^{\ge 0}\cap {^p}D^{\le 0}$. There exist perverse truncation functors
$$
{^p}\tau_{\le 0}:D^b_c(\Q_S)\to{^p}D^{\le 0},\ \ 
{^p}\tau_{\ge 0}:D^b_c(\Q_S)\to{^p}D^{\ge 0}
$$
that are adjoint to the inclusion functors. 

If $S$ is smooth and $L$ is a local system on $S$ then $L[d_S]\in\Perv(S)$. More generally, if $U\subset S$ is a smooth Zariski open subset and $L$ is a local system on $U$, then the {\em intermediate extension} $j_{!*}(L[d_S])$ is a perverse sheaf on $S$. We denote this perverse sheaf by $\IC_S(L)$, the intersection complex of $L$.  If $U = S_{\rm reg}$ is the smooth part of $S$ and $L = \Q_U$ we obtain the intersection complex  $\IC_S = j_{!*}(\Q_U[d_S])$. We define the intersection cohomology and homology groups of $S$ by
$$
IH^k(S) = H^{k-d_S}(\IC_S), \ \ IH_k(S) = H^{d_S-k}(\IC_S).
$$
and similarly we define intersection cohomology with values in a local sytem $L$ by $IH^k(S,L) = H^{k-d_S}(\IC_S(L))$. The simple objects of $\Perv(S)$ are of the form $i_*\IC_Z(L)$ where $Z\subset S$ is an irreducible subvariety and $L$ is an irreducible local sytem on a Zariski open subset of $Z$.

We define functors ${^p}H^i:D^b_c(\Q_S)\to\Perv(S)$ by 
$$
{^p}H^0(P) = ({^p}\tau_{\le 0}\circ {^p}\tau^{\ge 0})(P),\ \ {^p}H^i(P) = {^p}H^0(P[i]).
$$
In particular, if $f:X\to S$ is a morphism of algebraic varieties and $F\in D^b_c(\Q_X)$ we obtain the perverse direct image sheaves ${^p}R^if_*F = {^p}H^i(Rf_*F)\in\Perv(S)$. The decomposition theorem states that if $f:X\to S$ is a proper map of complex algebraic varieties then
$$
Rf_*\IC_X\simeq\oplus_i {^p}R^if_*\IC_X[-i].
$$
and that the perverse sheaves ${^p}R^if_*\IC_X$ are semisimple, i.e., they decompose as a direct sum 
$$
{^p}R^if_*\IC_X\simeq\oplus_{\alpha} i_{\alpha,*}\IC_{Z_{\alpha}}(L_{\alpha}).
$$
There exists a {\em perverse Leray spectral sequence} 
$$
{^p}E_2^{i,j} = H^i(S,{^p}R^jf_*\IC_X)\Rightarrow H^{i+j}(X,\IC_X) = IH^{i+j+d_X}(X)
$$
that degenerates at the $E_2$ term.

\begin{rmk}
\label{rmk: splitting}
The decomposition theorem implies that if $f:X\to S$ is proper and surjective, $\IC_S$ is a direct factor of $Rf_*\IC_X$. This means that there exist maps
$$
i:\IC_S\to Rf_*\IC_X,\ \ r:Rf_*\IC_X\to\IC_S
$$
such that $r\circ i\simeq{\rm id}_{\IC_S}$.
\end{rmk}

\subsection{Weak functoriality}

Although intersection cohomology is not functorial (it is only functorial with respect to certain classes of maps, for example placid maps) it has some weak functoriality properties as shown in \cite{BBFGK}. Recall that for an algebraic variety $X$ there exists a canonical morphism 
$\omega_X:\Q_X[d_X]\to\IC_X$
that induces maps 
$$
\omega_X:H^k(X,\Q) = H^{k-d_X}(\Q_X[d_X])\to H^{k-d_X}(\IC_X) = IH^k(X).
$$
The dual version is a map $\eta_X:\IC_X\to D_X[-d_X]$ that induces maps $IH_k(X) = H^{d_X-k}(\IC_X)\to H_k(X,\Q)$. 

In [loc.cit.] it is shown that if $f:X\to Y$ is a morphism of algebraic varieties there exist maps\footnote{The shift that we apply is different from the one in [loc.cit.], since we use a different convention for the definition of the intersection complex.}
\begin{eqnarray*}
\lambda_f:\IC_Y\to Rf_*\IC_X[d_Y-d_X] \\
\mu_f:Rf_*\IC_X\to\IC_Y[d_Y-d_X]
\end{eqnarray*}
that induce pullback and pushforward maps
\begin{eqnarray*}
f^{\#}: IH^k(Y)\to IH^k(X) \\
f_{\#}: IH_k(X)\to IH_k(Y)
\end{eqnarray*}
that make the diagrams
$$
\begin{array}{ccc}
H^k(Y) & \mapright{f^*} & H^k(X)  \\ 	
\mapdown{\omega_Y} & & \mapdown{\omega_X} \\
IH^k(Y) & \mapright{f^\#} & IH^k(X) 
\end{array}
\hskip 1cm
\begin{array}{ccc}
IH_k(X) & \mapright{f_\#} & IH_k(Y) \\
\mapdown{\eta_X} & & \mapdown{\eta_Y} \\
H_k(X) & \mapright{f_*} & H_k(Y) 	
\end{array}
$$
commute. We recall the construction of $\lambda_f$, following \cite{Weber}. Let 
$$
\sigma:\tilde Y\to Y
$$
be a resolution of singularities of $Y$. As shown in \cite[Remarque pp. 172-174]{BBFGK} there exists a commutative diagram
$$
\begin{array}{ccc}
\tilde X & \mapright{\tilde f} & \tilde Y \\
\mapdown{\tilde\sigma} & & \mapdown{\sigma} \\
X & \mapright{f} & Y
\end{array}
$$
such that 
\begin{itemize}
\item $\sigma$ and $\tilde\sigma$ are proper, surjective morphisms;
\item $\dim \tilde X = \dim X$ and $\dim\tilde Y = \dim Y$. 
\end{itemize}
Recall that by Remark \ref{rmk: splitting} there exist inclusion and retraction maps
$$
i_Y:\IC_Y\to\sigma_*\IC_{\tilde Y},\ \ r_Y:\sigma_*\IC_{\tilde Y}\to\IC_Y.
$$
The map $\lambda_f$ is obtained as the composition of the maps
\begin{eqnarray*}
	\IC_Y & \mapright{i_Y} & R\sigma_*\IC_{\tilde Y}=R\sigma_{Y,*}\Q_{\tilde Y}[d_Y] \\
	R\sigma_*\Q_{\tilde Y}[d_Y] & \mapright{R\sigma_*(\alpha_{\tilde f})} & 
		R\sigma_*R{\tilde f}_*\Q_{\tilde X}[d_Y]=Rf_*R_{\tilde\sigma_*\IC_{\tilde X}}[d_Y-d_X] \\
	Rf_*R{\tilde\sigma_*}\IC_{\tilde X}[d_Y-d_X] & \mapright{Rf_*(r_X)} & Rf_*\IC_X[d_X-d_Y].
\end{eqnarray*}
Note that the map $\mu_f$ is the (Verdier) dual of $\lambda_f$. 
In the special case where $X=\tilde Y$ and $f=\sigma:\tilde Y\to Y$ is a resolution of singularities, we have 
$$
\lambda_{\sigma} =i_Y:\IC_Y\to Rf_*\IC_{\tilde Y}, \ \
\mu_{\sigma} = r_Y:Rf_*\IC_{\tilde Y}\to\IC_Y.
$$
Hence the above construction of $\lambda_f$ shows that
\begin{equation}
\label{eq: sharp identity}
f^{\#} = \tilde\sigma_{\#}\circ {\tilde f}^*\circ \sigma^{\#}.	
\end{equation}

\begin{prop}
\label{prop:Gammasharp}	
The correspondences $\Gamma_i$ induce morphisms
$$
\Gamma_{i,\#}:R\pi_*\IC_{Z_i}\to Rf_*\IC_{\XX}
$$
for $i\in\{0,1\}$.
\end{prop}

\begin{proof}
	The map $\lambda_p:\IC_{Z_i}\to Rp_*\IC_{\Gamma}[-m]$ induces a map 
	$$
	R\pi_*\IC_{Z_i}\to R\pi_*Rp_*\IC_{\Gamma}[m] = Rf_*Rq_*\IC_{\Gamma}[m].
	$$
	When we compose this with $\mu_q:Rq_*\IC_{\Gamma}[-m]\to Rf_*\IC_{\XX}$ we obtain
	$$
	\Gamma_{i,\#}:R\pi_*\IC_{Z_i}\to Rf_*\IC_{\XX}.
	$$
\end{proof}

\begin{cor}
For all $k$ and $i\in\{0,1\}$ we have maps 
$$
\Gamma_{i,\#} = q_{\#}\circ p^{\#}:IH^k(Z_i)\to H^{k+2m}(\XX).
$$
\end{cor}

\subsection{Comparison of three correspondences}

Let $\sigma:\tilde Z_i\to Z_i$ be a resolution of singularities of $Z_i$. We write
$\tilde\Gamma_i = \Gamma_i\times_{Z_i} \tilde Z_i$.
Applying (\ref{eq: sharp identity}) to the diagram 
$$
\begin{array}{ccc}
\tilde\Gamma_i & \mapright{\tilde p} & \tilde Z_i \\
\mapdown{\tilde\sigma} & & \mapdown{\sigma} \\
\Gamma & \mapright{p} & Z_i
\end{array}
$$
we obtain
\begin{equation}
\label{eq: psharp}
p^{\#} = \tilde\sigma_{\#}\circ{\tilde p}^*\circ(\sigma)^{\#}.
\end{equation}
Note that $$
q_{\#}=q_*\circ\eta,\ \ \sigma_*=\tilde\sigma_{\#}\circ\eta
$$
since $\XX$ is smooth. Define 
$\tilde q = q\circ\tilde\sigma:\tilde\Gamma\to\XX$. The commutative diagram 
$$
\begin{array}{ccccc}
IH_k(\tilde\Gamma) & = & H_k(\tilde\Gamma)	& \mapright{\tilde q_*} & H_k(\XX) \\
\mapdown{\tilde\sigma_{\#}} & & \mapdown{\sigma_{*}} & & \Vert \\
IH_k(\Gamma) & \mapright{\eta} & H_k(\Gamma) & \mapright{q_*} & H_k(\XX)	 
\end{array}
$$
shows that 
\begin{equation}
\label{eq: qsharp}
q_{\#}\circ\tilde\sigma_{\#} =  q_*\circ\eta\circ\tilde\sigma_{\#} = q_*\circ\sigma_* = \tilde q_*. 
\end{equation}

\begin{prop}
\label{prop:commutativity}	
For all $k$ we have a commutative diagram
$$
\xymatrix{H^k(Z_i) \ar[r]^{\omega} \ar[dr]^{\Gamma_*} & IH^k(Z_i) \ar[r]^{\sigma^{\#}} \ar[d]^{\Gamma_{i,\#}} & H^k(\tilde Z_i) \ar[dl]^{\tilde\Gamma_*} \\
	& H^{k+2m}(\XX) \ar[d] & \\
	& H^{k+2m-2r}(X). &
}
$$
\end{prop}

\begin{proof} The commutativity of the left-hand diagram follows from the compatibility between the maps $p^{\#}$ (resp. $q_{\#}$) and $\omega$ (resp. $\eta$). The right-hand side follows from the identities (\ref{eq: psharp}) and (\ref{eq: qsharp}):
\begin{eqnarray*}
\tilde\Gamma_*\circ\sigma^{\#} & = & \tilde q_*\circ\tilde p^*\circ\sigma^{\#} \\
& = & q_{\#}\circ\tilde\sigma_{\#}\circ\tilde p^*\circ\sigma^{\#} \\
& = & q_{\#}\circ p^{\#}.
\end{eqnarray*}
The vertical map at the bottom is induced by the correspondence
$$
\xymatrix{
X\times\P^r \ar[r] \ar[d] & \XX \ar[r] \ar[d] & \P^{n+1}\times\P^r \ar[dl] \\
X \ar[r] & \P^{n+1}. &
}
$$	
\end{proof}

\section{Radon transformation}

Let $\P=\P(V)$ be a projective space of dimension $N$, and write $\P^{\vee} =\P^{\vee}$ for the dual projective space. Let $I\subset\P\times\P^{\vee}$ be the incidence correspondence with projections $p:I\to\P$ and $q:I\to\P^{\vee}$. The {\sl Radon transformation}
$$
\Phi:D^b_c(\P)\to D^b_c(\P^{\vee})
$$
is defined by
$$
\Phi(F^{\bullet}) = Rq_*(p^*F^{\bullet}[N-1]).
$$
Likewise we have the dual transformation $\Phi^{\vee}:D^b_c(\P^{\vee})\to D^b_c(\P)$. 

Note that if $P\in\Perv(\P)$ is a simple perverse sheaf, $p^*P[N-1]$ is a simple perverse sheaf on $I$ \cite{BBD} and the decomposition theorem shows that
$$
\Phi(P) = \oplus_{i=-N+1}^{N-1} \Phi^i(P)[-i]
$$
where $\Phi^i(P) = {^pH^i}(\Phi(P)) =  {^p R}^i q_*(p^*P[N-1])$. A perverse sheaf $P\in\Perv(\P)$ is called constant if there exists $m\ge 0$ such that $P\simeq\oplus^m \Q_{\P}[N]$. 

\begin{thm}[Brylinski]\cite{Brylinski} (see also \cite{Beilinson})
\label{thm:Brylinski}
\begin{enumerate}
\item[(i)] If $P\in\Perv(\P)$ then $\Phi^i(P) = {^p H^i}(\Phi(P))$ is a constant perverse sheaf for all $i\ne 0$.
\item[(ii)] Let $\CC$ (resp. $\CC^{\vee}$) be the subcategory of constant perverse sheaves on $\P$ (resp. $\P^{\vee}$). The functor $\Phi^0$ induces an equivalence of quotient categories 
$$
{\Perv(\P)/\CC} \mapright{\sim} {\Perv(\P^{\vee})/\CC^{\vee}}
$$
whose quasi--inverse is given by the functor $\Phi^{\vee,0} = {^P H^0}\circ \Phi^{\vee}$. 
\end{enumerate}
\end{thm}

\begin{rmk}
\label{rmk:reduced}
A perverse sheaf $P\in\Perv(\P)$ is called {\sl reduced} if it does not admit a subobject or quotient that is constant. To every object $P\in\Perv(\P)$ we can associate a reduced subquotient $P_{\rm red}\in\Perv(\P)$ \cite{KiehlWeissauer}. Two objects $P,Q\in\Perv(\P)$ become isomorphic in $\Perv(\P)/\CC$ if and only if $P_{\rm red}\simeq Q_{\rm red}$ in $\Perv(\P)$. A convenient way of reformulating part (ii) of the theorem is that the functor $\Phi^0$ induces a 1-1 correspondence $P\mapsto P^{\vee} = \Phi^0(P)_{\rm red}$ between reduced perverse sheaves on $\P$ and reduced perverse sheaves on $\P^{\vee}$; see \cite[II.3.3 and Remark, p. 210]{KiehlWeissauer}.
\end{rmk}

\subsection{ Application to the universal family}

The idea of applying the Radon transformation to the study of the universal family of divisors in a linear system $|L|$ goes back to Beilinson \cite{Beilinson}. 
Let $Y$ be a smooth projective variety of dimension $n+1$ equipped with a very ample line bundle $L$. Put $V=H^0(Y,L)^{\vee}$ and write $N=\dim(V)-1$. Let 
$\varphi_L:Y\hookrightarrow \P(V)$
be the embedding given by the linear system $|L|$. As before, we write $\P = \P(V)$, $\P^{\vee} = \P(V^{\vee})$ and consider the incidence correspondence $I\subset\P\times\P^{\vee}$. Note that the universal family
$$
\XX = \{(y,[s])\in Y\times\P^{\vee} | s(y) = 0\}
$$
associated to the linear system $|L|$ is the fibered product $Y\times_{\P} I$. Hence we obtain a diagram 
$$
\begin{array}{ccccc}
\XX & \mapright{\varphi'_L} & I & \mapright{q} & \P^{\vee} \\
\mapdown{p'} & & \mapdown{p} & & \\
Y & \mapright{\varphi_L} & \P & &
\end{array}
$$
where the left-hand square is cartesian. Write $f = q\circ \varphi'_L:\XX\to \P^{\vee}$ and
$d_{\XX} = \dim\XX = n+N$. 

By the decomposition theorem
$$
Rf_*\Q_{\XX}[d_{\XX}]\simeq \oplus_{i=-n}^{n} {^p R}^i f_*(\Q_{\XX}[d_{\XX}])[-i].
$$
Let $\varphi:Y\times\P^{\vee}\to\P^{\vee}$ the projection map.  
The perverse Lefschetz hyperplane theorem shows that
$
{^p R}^if_*(\Q_{\XX}[d_{\XX}])
$
is isomorphic to the constant perverse sheaf $(R^i\varphi_*\Q)[N]$ for all $i\ne 0$. The remaining term ${^p R}^0 f_*(\Q_{\XX}[d_{\XX}])$ splits as a sum of simple perverse sheaves:
$$
{^p R}^if_*(\Q_{\XX}[d_{\XX}])\simeq \oplus_{\alpha} \IC_{Z_{\alpha}}(L_{\alpha})
$$
where $Z_{\alpha}$ is an irreducible subvariety of $\P$ and $L_{\alpha}$ is a locally constant sheaf on a Zariski open subset of $Z_{\alpha}$. Write 
$$
E_{0,i}(f) = \oplus_{\alpha:\codim(Z_{\alpha}) = i} \IC_{Z_{\alpha}}(L_{\alpha}).
$$

We say that the {\em vanishing cycles are non trivial} if $(R^{n}f_*\Q)_{\rm var}\ne 0$ (equivalently, the cohomology class $[\delta_t]\in H^{n}(X_t)$ associated to any vanishing cycle in a Lefschetz pencil is nontrivial). 
\begin{thm}\cite{BFNP}
\label{thm:BFNP}	
Let $U = \P^{\vee}\setminus\Delta$ be the complement of the discriminant locus of $f$.
\begin{enumerate}
\item[1)] If the vanishing cycles are trivial, then
$E_{0,0}(f) = (R^{n}\varphi_*\Q)[N] \oplus \IC(M)$, where $M = (R^{n}f_*\Q)_{\rm var}$, and $E_{0,1}(f)=0$.
\item[2)] If the vanishing cycles are trivial, then $E_{0,0}(f) = (R^{n}\varphi_*\Q)[N]$ and $E_{0,1}(f) = IC_{\Delta}(M')$ where $M'$ is a rank one local system on an open subset of $\Delta$.
\end{enumerate}
\end{thm}

\begin{prop} 
\label{prop:Radon}	
Let $Y_L\subset\P$ be the image of $\varphi_L$, with embedding $i:Y_L\hookrightarrow\P$. Then 
$Rf_*\Q_{\XX}[d_{\XX}]\simeq \Phi(i_*\Q_{Y_L}[n+1])$
where $\Phi$ is the Radon transformation.
\end{prop}

\begin{proof}
Since the maps $\varphi_L$ and $\varphi'_L$ are proper, the proper base change theorem shows that 
$p^*\circ\varphi_{L,*} = \varphi'_{L,*}\circ (p')^*$. Hence 
\begin{eqnarray*}
Rf_*\Q_{\XX}[d_{\XX}]  & = & Rq_*(R \varphi'_{L,*}((p')^*\Q_Y[n+N]) \\
& = & Rq_*(p^*(\varphi_{L,*}(\Q_Y[n+1])[N-1]) \\
& = & \Phi(i_*\Q_{Y_L}[n+1]).
\end{eqnarray*}
\end{proof}

\begin{cor} 
\label{cor:decomposition}	
Let $f:\XX\to S$ be a regular quadric bundle of relative dimension $n$ over a base of dimension $r$ obtained by pullback from the universal family of quadrics in $\P^{n+1}$. 
\begin{enumerate}
\item[a)] If $n=2m$ then
$$
Rf_*\Q_{\XX}[d_{\XX}]\simeq\IC(M_0)\oplus{\displaystyle\oplus}_{i=-m}^{m} \Q[r-2i]
$$
where $M_0$ is the local system $(R^{2m}f_*\Q)_{\rm var}$ on $S\setminus\Delta$.
\item[b)] If $n=2m-1$ then
$$
Rf_*\Q_{\XX}[d_{\XX}]\simeq\oplus_{i=-m+1}^{-1} \Q[r-1+2i]\oplus\IC_{\Delta}(M_1)\oplus_{i=1}^{m-1} \Q[r-1-2i]
$$
where $M_1$ is the rank one local system $(R^{2m}f_*\Q)_{\rm var}$ supported on a Zariski open subset of $\Delta$.	
\end{enumerate}
\end{cor}


\begin{proof}
We first prove the result for the universal family of quadrics $f:\XX_{\rm univ}\to\P^N$. Let $Y\subset\P^N$ be the image of the Veronese embedding $\P^{n+1}\hookrightarrow\P^N$. Write $P = i_*\Q_Y[n+1]\in\Perv(\P)$. As $P$ is an irreducible perverse sheaf, its dual $P^{\vee} = \Phi(P)_{\rm red} = (\oplus_i E_{0,i}(f))_{\rm red}$ is irreducible by Theorem \ref{thm:Brylinski}. By Proposition \ref{prop:Radon} we have
$\Phi(P)_{\rm red} = (\oplus_i E_{0,i}(f))_{\rm red}$,  
hence there exists a unique index $i_0$ such that $E_{0,i_0}(f)\ne 0$. Theorem \ref{thm:BFNP} shows that $i_0\in\{0,1\}$. Hence $E_{0,i}(f) = 0$ for all $i\ge 2$ and the result follows.
If $i:S\to\P^N$ is a normally nonsingular inclusion of codimension $c$, then the functor $i^*[-c]$ is t-exact for the perverse t-structure and the decomposition for $f:\XX\to S$ follows by base change as in \cite[Lemma 4.3.8]{CMENS}.
\end{proof}

\section{Main result}

\subsection{Quadric bundles of even relative dimension}

\begin{thm}
\label{thm:main even}
Let $f:\XX\to S$ be a regular quadric bundle of even relative dimension $n$ over a base of dimension $r$ obtained by pullback of the universal family of quadrics in $\P^{n+1}$. Then the correspondence $\Gamma_0$ induces an isomorphism
$$
\Gamma_{0,\#}:IH^r_{\rm var}(Z_0)\mapright{\sim} H^{n+r}_{\rm var}(\XX,\Q)
$$
where $IH^r(Z_0)_{\rm var}$ is the cokernel of the composed map 
$$
H^r(S,\Q)\mapright{\pi^*}H^r(Z_0)\mapright{\omega}IH^r(Z_0)
$$
\end{thm}

\begin{proof}
Write $Y = \P^{n+1}\times S$ and let $\iota:\XX\to Y$ be the inclusion map. Since $\iota^*\IC_Y[-1]\simeq\IC_{\XX}$ we have a map 
$$
\lambda:\IC_Y[-1]\to\iota_*\IC_{\XX}
$$
that induces the restriction maps $\iota^*:H^k(Y)\to H^k(\XX)$ on cohomology ($\lambda$ is the map $\lambda_{\iota}$ from the previous section). Consider the commutative diagram
$$
\xymatrix{
\XX \ar[r]^{\iota} \ar[dr]^{f} & Y \ar[d]^{\varphi} \\
& S.
}
$$
The map $\lambda$ induces maps
$R\varphi_*(\IC_Y[-1])\to Rf_*\IC_{\XX}$
and 
$$
\lambda_j:{^p}R^{j-1}\varphi_*\IC_Y\to {^p}R^jf_*\IC_{\XX}
$$
that are isomorphisms for $j\le -1$ and injective for $j=0$  (this is the perverse relative Leschetz hyperplane theorem, cf. \cite[Thm. 2.6.4]{dCM}).
Note that the image of $\lambda_j$ is precisely the constant part of ${^p}R^jf_*\IC_{\XX}$, so the cokernel is the reduced quotient.
These maps are compatible with the perverse Leray spectral sequences
\begin{eqnarray*}
{^p}E_2^{i,j}(\varphi) & = & H^p(S,{^p}R^j{\varphi}_*\IC_Y)\Rightarrow H^{i+j}(\IC_Y[-1]) = H^{i+j+d_Y-1}(Y) \\
{^p}E_2^{i,j}(f) & = & H^p(S,{^p}R^f_*\IC_{\XX})\Rightarrow H^{i+j}(\IC_{\XX}) = H^{i+j-1+d_{\XX_S}}(\XX). 	
\end{eqnarray*}
Since the perverse Leray spectral sequences degenerate at $E_2$, Corollary \ref{cor:decomposition} implies that 
$$
H^{n+r}_{\rm var}(\XX)\cong H^0(S,\IC_S(M_0)).
$$
with $M_0 = (R^{n-1}f_*\Q)_{\rm var}$.

By \cite[Lemma 4.2.6]{CDN} we have a decomposition 
$R\pi_*\IC_{Z_0}\simeq\Q_S\oplus\IC_S(L_0)$.
The morphism
$$
\Gamma_{0,\#}:R\pi_*\IC_Z\to Rf_*\IC_{\XX_S}
$$
constructed in Propositon \ref{prop:Gammasharp} is compatible with the perverse truncation and induces a map
$$
\begin{array}{ccc}
({^p}R^0\pi_*\IC_{Z_0})_{\rm red} & \mapright{\Gamma_{0,\#}} &  ({^p}R^0f_*\IC_{\XX})_{\rm red} \\
\Vert & & \Vert \\
\IC_S(L_0) & \mapright{} & \IC_S(M_0).
\end{array}
$$
Since $\Gamma_{0,\#}$ is compatible with $\Gamma_*$, the restriction to $U = S\setminus\Delta$ is the isomorphism $\Gamma_{0,*}|_U:L_0\mapright{\sim}M_0$ of Proposition \ref{prop:iso loc even}. As $\Hom_U(L_0,M_0)\cong\Hom_S(\IC_S(L_0),\IC_S(M_0))$ \cite[Prop. 9.2, p.144]{Borel}, $\Gamma_{0,\#}:\IC_S(L_0)\to\IC_S(M_0)$ is an isomorphism. 
Hence $\Gamma_{0,\#}$ induces an isomorphism
$$
IH^r(Z_0)_{\rm var} = H^0(S,\IC_S(L_0))\to H^0(S,\IC_S(M_0)) = H^{n+r}_{\rm var}(\XX).
$$
\end{proof}

\subsection{Quadric bundles of odd relative dimension}

\begin{thm}
\label{thm:main odd}
Suppose that $f:\XX\to S$ is a regular quadric bundle of odd relative dimension $n$ over a base of dimension $r$ obtained by pullback from the universal family of quadrics in $\P^{n+1}$. Let $IH^{r-1}_{\rm var}(Z_1)$ be the cokernel of the map
$\pi^{\#}:IH^{r-1}(\Delta)\to IH^{r-1}(Z_1)$.
Then the correspondence $\Gamma_1$ induces an isomorphism
$$
\Gamma_{1,\#}:IH^{r-1}_{\rm var}(Z_1)\mapright{\sim} H^{n+r}_{\rm var}(\XX,\Q)
$$
\end{thm}

\begin{proof}
Since the proof is entirely similar to the proof of Theorem \ref{thm:main even}, we only indicate the differences. If $n$ is odd, \cite[Lemma 4.2.6]{CDN} implies that
$$
R\pi_*\IC_{Z_1}\simeq\IC_{\Delta}\oplus\IC_{\Delta}(L_1).
$$
The correspondence $\Gamma_{1,\#}$ induces a morphism
$$
{^p}R^0\pi_*\IC_{Z_1}\simeq \IC_{\Delta}\oplus\IC_{\Delta}(L_1)\to {^p}R^0f_*\Q_{\XX}[d_{\XX}]\simeq\IC_{\Delta}(M_1),
$$
and since $\IC_{\Delta}(L_1)\cong\IC_{\Delta}(M_1)$ by Proposition \ref{prop: iso loc odd} the correspondence $\Gamma_{1,\#}$ induces an isomorphism
$$
IH^{r-1}_{\rm var}(Z_1) = H^0(\IC_{\Delta}(L_1))\cong H^0(\IC_{\Delta}(M_1) = H^{n+r}_{\rm var}(\XX).
$$ 
\end{proof}

\subsection{Applications}

\medskip
\begin{thm} 
Let $X=V(Q_0,\ldots,Q_r)\subset\P^{n+1}$ be a smooth complete intersection of  quadrics such that the associated quadric bundle is regular. Then 
$$
\Gamma_{i,\#}:IH^{r-i}_{\rm var}(Z_i)\to H^{n-r}_{\rm var}(X)
$$
is an isomorphism for $n\equiv i(2)$.
\end{thm}

\begin{rmk}
If $r\le 2$, the double coverings $Z_0$ and $Z_1$ are smooth and we obtain isomorphisms
$$
H^{r-i}(Z_i)_{\rm var}\mapright{\sim} H^{n-r}_{\rm var}(X,\Q)
$$
for $n\equiv i(2)$. This gives a unified proof of the results of Beauville, Mukai, O'Grady and Tyurin mentioned in the introduction (with $\Q$-coefficients).	
\end{rmk}

Recall the statement of the generalised Hodge conjecture: we say that ${\rm GHC}(X,n,p)$ holds if every sub-Hodge structure of $H^n(X)$ of level $\le p$ is supported on a subvariety $Z\subset X$ of codimension $p$, i.e., $V\subseteq\im(H^n_Z(X)\to H^n(X))$.
\begin{cor}
The generalised Hodge conjecture ${\rm GHC}(X,n-r,r-i)$ holds for a smooth complete intersection of quadrics $X = V(Q_0,\ldots,Q_r)\subset\P^{n+1}$ if the associated quadric bundle is regular.
\end{cor}

\begin{proof}
As the map
$$
\Gamma_{i,\#}:IH^{r-i}_{\rm var}(\tilde Z_i)\to H^{n-r}_{\rm var}(X) 
$$
is an isomorphism if $n\equiv i(2)$, Proposition \ref{prop:commutativity} shows that $\tilde\Gamma_{i,*}$ induces a surjective map from $H^{r-i}(\tilde Z_i)$ to $H^{n-r}_{\rm var}(X)$. Hence the generalised Hodge conjecture ${\rm GHC}(X,n-r,r-i)$ holds. 
\end{proof}


\begin{rmk}
There exist examples of non-regular quadric bundles for which $\Gamma_{i,\#}$ is not surjective. Consider a smooth complete intersection $X=V(Q_0,Q_1,Q_2)\subset\P^{2m+1}$ of three quadrics
of {\em diagonal type}
$$
Q_i = \sum_j a_{ij}X_j^2
$$	
for a generic matrix $\{a_{ij}\}$. The discriminant $\Delta_1$ is an arrangement of $2m+2$ lines in $\P^2$, and $\Delta_2$ is a set of $\mu={2m+2\choose 2}$ points. The double covering $Z_0\to\P^2$ is a surface with $\mu$ ordinary double points. Terasoma \cite[4.4, Example 3]{Terasoma} has shown that the cokernel of the map
$$
\Gamma_{0,*}:H^2_{\rm var}(Z_0)\to H^{2m-2}_{\rm var}(X)
$$
is isomorphic to $\oplus_{0\le i<j\le 2m+1}\Q$. Since $Z_0$ is a surface with isolated singularities, the map $H^2(Z_0)\to IH^2(Z_0)$ is surjective. Hence Proposition \ref{prop:commutativity} shows that $\Gamma_{0,\#}$ cannot be surjective in this example.
Note that the cokernel of $\Gamma_{0,*}$ is governed by skyscraper sheaves supported on $\Delta_2$, a contribution that one would not see for regular quadric bundles (where $\Delta_2$ has codimension three).

A similar result holds for complete intersections of four quadrics $X = V(Q_0,Q_1,Q_2,Q_3)\subset\P^{2m}$ [loc.cit., Example 4]. In this case the cokernel of 
$$
\Gamma_{1,*}:H^2_{\rm var}(Z_1)\to H^{2m-4}_{\rm var}(X)
$$
is isomorphic to $\oplus_{i<j<k}\Q$ and comes from skyscraper sheaves supported on $\Delta_3$. Since $Z_1$ is a surface with isolated singularities, this implies that $\Gamma_{1,\#}$ is not surjective.
\end{rmk}

\subsection{Complete intersections of four quadrics in odd-dimensional projective space}
\ 

\medskip
A special case is smooth the complete intersection of four quadrics 
$$
X = V(Q_0,Q_1,Q_2,Q_3)\subset\P^{2m+1}
$$
with $m\ge 3$. If the associated quadric bundle $\XX\to\P^3$ is regular, then the discriminant locus is a surface $\Delta\subset\P^3$ whose singular locus $\Sigma$ consists of ordinary double points. The double covering $Z_0\to\P^3$ is a double solid whose singularities are ordinary points lying above $\Sigma$.


Since $\Gamma_{0,*}$ is a morphism of Hodge structures of type $(m-3,m-3)$ we have $W_2H^3(Z_0)\subseteq\ker(\Gamma_{0,*})$. Consider the induced map
$$
\overline{\Gamma}_{0,*}:\Gr_3^W H^3(Z_0)\to H^{2m-3}(X).
$$

\begin{thm} Let $Z_0\to\P^3$ be the doube solid associated to a regular quadric bundle coming from a web of quadrics in projective space, and let $\sigma:\tilde Z_0\to Z_0$ be the blowup of the double points of $Z_0$. Then all the maps in the diagram 
$$
\xymatrix{
\Gr^3_W H^3(Z_0) \ar[r]^{\omega} \ar[dr]^{\overline{\Gamma}_{0,*}} & IH^3(Z_0) \ar[r]^{\sigma^{\#}} \ar[d]_{\Gamma_{0,\#}} & H^3(\tilde Z_0) \ar[dl]^{\tilde\Gamma_{0,*}} \\
& H^{2m-3}(X)	&
}
$$
are isomorphisms.	
\end{thm}
	
\begin{proof} Since $\Gamma_{0,\#}$ is an isomorphism, it suffices to show that
$$	
\Gr^3_W H^3(Z_0)\cong IH^3(Z_0)\cong H^3(\tilde Z_0).
$$
For the first isomorphism, note that since $Z_0$ has isolated singularities
$IH^3(Z_0)\cong\im H^3(Z)\to H^3(Z_0\setminus\Sigma)$. By \cite{Steenbrink} we have $\Gr^3_W H^3_{\Sigma}(Z_0) = 0$, hence the exact sequence
$$
H^3_{\Sigma}(Z_0)\to H^3(Z_0)\to H^3(Z_0\setminus\Sigma)\to H^4_{\Sigma}(Z_0)
$$
induces an isomorphism $\Gr^3_W H^3(Z_0)\cong IH^3(Z_0)$.
	
For the second isomorphism we use the results of \cite{Leiden}, where the decomposition theorem for the map $\sigma:\tilde Z_0\to Z_0$ was worked out. The authors prove that
\begin{eqnarray*}
{^p}R^{-1}\sigma_*\Q[3] & \cong & H_4(E)_{\Sigma} \\
{^p}R^{0}\sigma_*\Q[3] & \cong  & \IC_{Z_0}\oplus H_3(E)_{\Sigma} \\
{^p}R^{1}\sigma_*\Q[3] & \cong & H^4(E)_{\Sigma}
\end{eqnarray*} 
where $E$ is the exceptional divisor, and $V_{\Sigma}$ denotes the skyscraper sheaf with fiver $V$ over $\Sigma$. In our case $E$ is a disjoint union of $\mu = \#\Sigma$ quadric surfaces, so we obtain
\begin{equation}
\label{eq:dC-M}	
R\sigma_*\Q[3]\cong i_*\Q_{\Sigma}^{\mu}[1]\oplus\IC_Z\oplus i_*\Q_{\Sigma}^{\mu}[-1]
\end{equation}
where $i:\Sigma\to Z_0$ is the inclusion. Taking hypercohomology, we obtain
$H^3(\tilde Z_0)\cong IH^3(Z_0)$. 
\end{proof}	

\begin{rmk}\ 
\begin{enumerate}
\item Let ${\hat Z}_0\to Z_0$ be a small resolution. Then $H^3(\tilde Z_0)\cong H^3({\hat Z}_0)$ \cite[p. 12]{Werner}; this gives back Addington's result \cite{Addington}.
\item Once we know that the map $\tilde\Gamma_{0,*}$ is surjective, we can check the isomorphism by a dimension calculation. Deligne's formula for the Hodge polynomial of a complete intersection \cite[Expos\'e XI, Cor. 2.4 (ii)]{Deligne} implies that
\begin{eqnarray*}
h^{m-2,m-1}(X) & = & \frac{5}{2}(m-2)^3 + \frac{31}{2}(m-2)^3+30m-43 \\
h^{m-3,m}(X) & = & \frac{1}{6}((m-2)^3+3(m-2)^2+2m-4).
\end{eqnarray*}
The Hodge numbers of $\tilde Z_0$ can be computed using the formulas of Clemens \cite{Clemens}. The double solid is branched over a surface of degree $2m+4$ and has $\mu = {{2m+3}\choose 3}$ ordinary double points. Using [loc.cit.] we obtain
\begin{eqnarray*}
h^{1,2}(\tilde Z_0) & = & {{3m+2}\choose 3}-4{{m+1}\choose 3}-\mu+\delta \\
h^{0,3}(\tilde Z_0) & = & {m\choose 3} 
\end{eqnarray*}
where $\delta = b-4(Z_0)-b_2(Z_0)$ is the {\em defect} of the double solid $Z_0$. Using Eagon-Northcott resolutions, Addington  \cite[Prop. 4.0.4]{Addington} has shown that $\delta=0$ in this case, hence the result follows by checking that both formulas give the same result. In the simplest cases $m=3$ and $m=4$ we get
\begin{eqnarray*}
X=V(2,2,2,2)\subset\P^7 &: & h^{1,2}(X) = h^{1,2}(\tilde Z_0)=65, h^{0,3}(X) = h^{0,3}(\tilde Z)=1 \\
X=V(2,2,2,2)\subset\P^9 & : & h^{2,3}(X) = h^{1,2}(\tilde Z_0)=159,\ h^{1,4}(X) = h^{0,3}(\tilde Z_0)=4.
\end{eqnarray*}
\end{enumerate}	
\end{rmk}

\bibliographystyle{amsalpha}
\bibliography{Collino}

\end{document}